\documentclass{amsart}
\usepackage[psamsfonts]{amssymb}
\usepackage{amsmath,amsthm}

\usepackage{enumerate}
\usepackage{multirow}

\newtheorem{thm}{Theorem}[section]
\newtheorem{cor}[thm]{Corollary}
\newtheorem{lem}[thm]{Lemma}
\newtheorem{prop}[thm]{Proposition}
\newtheorem{conj}[thm]{Conjecture}

\theoremstyle{remark}

\def\sph{\mathbb{S}^{d-1}}
\def\f{\frac}

 \def\a{{\alpha}} 
 \def\b{{\beta}}
 \def\g{{\gamma}}
 
 \def\t{{\theta}}
 \def\l{{\lambda}}
 \def\d{{\delta}}
 
 \def\s{{\sigma}}
 \def\la{{\langle}}
 \def\ra{{\rangle}}

 \def\CL{{\mathcal L}}

 \def\CC{{\mathbb C}}
 \def\HH{{\mathbb H}}
  
 \def\NN{{\mathbb N}}
 \def\PP{{\mathbb P}}
 
 \def\RR{{\mathbb R}}
  \def\SS{{\mathbb S}}

 \def\cc{{\textnormal{c}}}
 \def\dd{{\textnormal{d}}}

\begin{document}
 
\title[Positive definite functions on the unit sphere]
{Positive definite functions on the unit sphere and integrals of Jacobi polynomials}

\author{Yuan Xu}
\address{Department of Mathematics\\ University of Oregon\\
    Eugene, Oregon 97403-1222.}\email{yuan@uoregon.edu}

\thanks{The author was supported in part by NSF Grant DMS-1510296.}

\date{\today}

\begin{abstract} 
For $\a, \b \in \NN_0$ and $\max\{\a,\b \} >0$, it is shown that the integrals of the Jacobi polynomials  
\begin{equation*}
  \int_0^t (t-\theta)^\delta P_n^{(\alpha-\frac12,\beta-\frac12)}(\cos \theta) \left(\sin \tfrac{\theta}2\right)^{2 \alpha} 
       \left(\cos \tfrac{\theta}2\right)^{2 \beta} d\theta > 0
\end{equation*}
for all $t \in (0,\pi]$ and $n \in \mathbb{N}$ if $\delta \ge  \alpha + 1$ for $\a,\b \in  \mathbb{N}_0$ and $\max\{\a,\b\} > 0$. 
This proves a conjecture on the integral of the Gegenbauer polynomials in \cite{BCX} that implies the strictly positive 
definiteness of the function $\theta \mapsto (t - \theta)_+^\delta$ on the unit sphere $\SS^{d-1}$ for 
$\delta \ge \lceil \frac{d}{2}\rceil$ 
and the P\'olya criterion for positive definite functions on the sphere $\SS^{d-1}$ for $d \ge 3$. Moreover, the positive 
definiteness of the function $\theta \mapsto (t - \theta)_+^\delta$ is also established on the compact two-point 
homogeneous spaces. 
\end{abstract}

\maketitle

\section{Introduction}
\setcounter{equation}{0}

Let $\sph$ be the unit sphere in the Euclidean space $\RR^d$. Let $\dd(\cdot,\cdot)$ be the usual geodesic distance on $\sph$,
so that $\dd(x,y) :=\arccos \la x,y\ra$. A continuous function $g: [0,\pi]\mapsto \RR$ is called {\it positive definite} on $\sph$ 
if for any $N \in \NN$ and any set of $N$ distinct points $X_N :=\{x_1,\ldots,x_N\}$ in $\sph$, the $N\times N$ matrix
$$
    g [X_N]: = [ g (\dd (x_i,x_j))]_{i,j = 1}^N
$$ 
is nonnegative definite, and it is said to be {\it strictly positive definite} if the matrix $g [X_N]$ is positive definite. 

Let $C_n^\l$ denote the standard Gegenbauer polynomial of degree $n$. In his classical paper \cite{Sc42}, Schoenberg 
characterized the positive definite functions on $\sph$ as those functions of the form 
\begin{equation} \label{eq:g=sum}
   g(\t) = \sum_{k=0}^\infty a_k C_k^{\f{d-2}{2}}(\cos \t), \qquad a_k \ge 0,
\end{equation}
in which the series converges when $\t =0$. Motivated by interpolating scatted data on the sphere, the strictly positive
definite functions were first considered in \cite{XC}, where a sufficient condition of all $a_k > 0$ in \eqref{eq:g=sum} was 
established, and later characterized in \cite{CMS} as functions with $a_k > 0$ for infinitely many even and infinitely many
odd indices in \eqref{eq:g=sum} in the case $d \ge 3$, whereas the case $d =2$ was later established in
\cite{MOP}. A self-contained proof is given in \cite[Section 14.2]{DaiX}.

The characterization through the coefficients of the orthogonal expansions in Gagenbauer polynomials is difficult to verify
for a particular example. Confronting this difficulty, a Poly\'a type of criterion was established in \cite{BCX} for $d \le 8$ 
and shown to hold for all $d > 8$ should the function
\begin{equation*} 
    g_{t,\d} (\t) := (t -\t)_+^\d = \begin{cases} (t- \t)^\d, & \t \le t \\ 0, & \t > t \end{cases}
\end{equation*}
be proved to be strictly positive definite on $\sph$ when $\d \ge \lceil \frac{d}2 \rceil$. The latter holds if the following more
general conjecture, stated in \cite[Conjecture 1.4]{BCX}, is established: 

\begin{conj}\label{conjecture}
Let $\d > 0$, $\l > 0$ and $n \in \NN_0$. For every $0 < t \le \pi$, define
\begin{equation}\label{eq:FnG}
   F_n^{\l,\d} (t):= \int_0^t (t-\t)^\d C_n^\l (\cos \t) (\sin \t)^{2\l} d\t.
\end{equation}
Then $F_n^{\l,\d} > 0$ for all $t \in (0,\pi]$ if $\d \ge \l+1$. 
\end{conj} 
 
The purpose of this paper is to prove the sufficiency of the conjecture when $\l$ is a nonnegative integer, 
which establishes, in particular, the P\'olya criterion for the strictly positive definite functions in \cite{BCX} 
for all dimensions. 

When $\l$ and $\d$ are both integers, the integral $F_n^{\l,\d}$ can be written as a finite sum of polynomial 
and trigonometric functions; see Lemma \ref{lem:sin-cos}. The mixture of both types of elementary functions, 
however, is difficult to work with. The finite sum expressions were used to show that the conjecture
holds for $\l = 1, 2, 3$ in \cite{BCX}, for which a computer algebra system was used to handle the cases of 
small $n$. The approach does not seem to be extendable to the general case. In this work, we shall introduce 
an integral of the Jacobi polynomials that is more general than $F_n^{\l,\d}$. The richer structure of 
the Jacobi polynomials allows us to connect the problem to a known positive integral of the Bessel functions 
(\cite{GG75,GG}). The positivity of the integrals of the Jacobi polynomials also allows us to establish the strictly 
positive definiteness of the function $g_{t,\d}$ on other compact two-point homogeneous spaces.

There has been renewed interest in positive or strictly positive definite functions due to application in spatial 
statistics and approximation theory; see, for example, the recent survey \cite{G}, as well as \cite{BC, BC2, BCX,Ma,Z}. 
Among known examples, the function $g_{t,\d}$ has the simplest structure and, for a fixed $x_0 \in \sph$, the 
function $x \mapsto g_{t,\d}(\dd(x,x_0))$ has the compact support on $\sph$ and its support is precisely the 
spherical cap $\cc(x,\t): = \{x \in \sph: d(x,x_k) \le t\}$. The latter fact is of interests in interpolating scatted data. 

The paper is organized as follows. We state and discuss our main results in Section 2, including those on
compact two-point homogeneous spaces. The proofs are given in Section 3, where further properties
of the positive integrals of the Jacobi polynomials will also be discussed. 

\section{Main Results}
\setcounter{equation}{0}

We use the standard notation of $C_n^\l$ for the Gegenbauer polynomials and $P_n^{(\a,\b)}$ for the Jacobi
polynomials. For $\a,\b > - \f12$, we define the integral 
\begin{equation}\label{eq:Fn^J}
  F_n^{(\a,\b),\d} (t) := \int_0^t (t-\t)^\d P_n^{(\a-\f12,\b-\f12)}(\cos \t) \left(\sin \tfrac{\t}2\right)^{2 \a} 
       \left(\cos \tfrac{\t}2\right)^{2 \b} d\t.
\end{equation}
These integrals are generalizations of the integrals $F_n^{\l,\d}$ defined in \eqref{eq:FnG}. Indeed, using the well--known 
relation \cite[(4.7.1)]{Sz}  
\begin{equation} \label{eq:Cn=Pn}
C_n^\l (t) =  \frac{(2 \l)_n}{(\l+\f12)_n} P_n^{(\l-\f12,\l - \f12)}(t), \qquad \l > -\f12,  
\end{equation}
and $\sin \t = 2 \sin \f{\t}{2}\cos \f{\t}{2}$, we conclude immediately that 
\begin{equation} \label{eq:Cn=Pn2}
  F_n^{\l, \d} (t) = \frac{2^{2\l} (2 \l)_n}{(\l+\f12)_n}  F_n^{(\l,\l),\d} (t). 
\end{equation}

Our main effort lies in proving the following theorem: 

\begin{thm} \label{thm:main}
If $\a, \b \in \NN_0$ are not both zero, then $F_n^{(\a,\b),\d}(t) > 0$ for $t > 0$ and $n \in \NN_0$ if $\d \ge \a+1$. 
Moreover, $F_n^{(0,0),\d}(t) \ge 0$ for  $t > 0$ and $n \in \NN_0$ if $\d \ge 1$. In particular, Conjecture \ref{conjecture}
holds true for $\l \in \NN$. 
\end{thm}

The positivity of $F_n^{(\a,\b),\d}(t)$ is also established for some non-integer values of $\a,\b$; see 
Theorem \ref{thm:main2} in the next section. 

As shown in \cite{BCX}, the affirmative of Conjecture \ref{conjecture} for integer values of $\l$ yields 
the following corollary. 

\begin{cor} \label{cor:SPDF}
For $d =3,4,5,\ldots$ and $0 < t \le \pi$, the function $\t \mapsto (t-\t)_+^\d $ is strictly positive definite on the
sphere $\sph$ provided $\d \ge \lceil \frac{d}{2} \rceil$.
\end{cor}

This corollary was established in \cite{BCX} for $d \le 8$ with a direct but much more difficult proof. As shown in 
\cite[Theorem 1.2]{Z}, the condition $\d \ge  \lceil \frac{d}{2} \rceil$ is necessary for all even dimensions. Furthermore, it is proved
in \cite{BCX} that this corollary implies the following P\'olya criterion \cite[Theorem 1.3]{BCX}: 

\begin{thm} 
Let $d \in \NN$ and $d \ge 3$ and let $\l = \lceil \frac{d-2}{2} \rceil$. Let $g: [0,\pi] \mapsto \RR$ be a continuous 
function such that $g \in C^\l[0,\pi]$ and its derivatives satisfy
\begin{enumerate}[ (i)]
\item  $(-1)^\l g^{(\l)}$ is convex,
\item $g^{(j)}(\pi) =0$ for $0 \le j \le \l +1$ and $g^{(\l+1)}(0)$ is finite. 
\end{enumerate}
Then $g$ is a positive definite function on $\sph$. If, in addition, that $g^{(\l)}$ is not
a linear polynomial, then $g$ is a strictly positive definite function on $\sph$. 
\end{thm}

The condition (ii) is weakened in \cite[Theorem 6]{G} when $d$ is an even integer, where $g$ is assumed to be the restriction 
to $[0,\pi]$ of a continuous function on $[0,\infty)$ that satisfies (i) and $g(0) =1$ and $\lim_{t \to \infty} g(t) =0$. 

We now consider the compact two-point homogeneous spaces. These are spaces that have the property that a group of 
motions exists which takes any points $(x_1,y_1)$ to $(x_2,y_2)$ when $\mathrm{dist}(x_1,y_1) = \mathrm{dist} (x_2,y_2)$. 
Besides the unit sphere $\sph$, there are four other classes of such spaces: real projective space $\PP^d(\RR)$, $d =2,3,\ldots$, 
the complex projective space $\PP^d(\CC)$, $d=4,6, ...$, the quaternionic projective spaces $\PP^d(\HH)$, $d = 8, 10, \ldots$, 
and the Cayley projective plane $\PP^{16}(\mathrm{Cay})$. The positive, and strictly positive, definite functions on these 
spaces are defined similarly as those on the unit sphere. By \cite{Ga}, the positive definite functions on these spaces are given by 
\begin{equation} \label{eq:pdf2pt}
    f(\t) = \sum_{n=0}^\infty a_n P_n^{(\a,\b)}(\cos n \t), \qquad a_n \ge 0,  
\end{equation}
where the infinite series converges at $\t = 0$ and $\mu$ is a constant depending on the diameter of the space, and 
$\a,\b$ are specified by 
\begin{enumerate}
\item $\PP^d(\RR)$: $\a = \b = \tfrac{d-2}{2}$ and $a_{2k+1} =0$ for all $k$;
\item $\PP^d(\CC)$: $\a = \tfrac{d-2}{2}$ and $\b = 0$; 
\item $\PP^d(\HH)$: $\a = \tfrac{d-2}{2}$ and $\b = 1$; 
\item $\PP^{16}(\mathrm{Cay})$: $\a = 7$, $\b = 3$. 
\end{enumerate}
Furthermore, the strictly positive definite functions on these spaces are studied in \cite{BM, BonM}, where results similar to those 
on the unit sphere are established. In particular, if all $a_n > 0$ (or $a_{2n} > 0$ for $\PP^d(\RR)$), then the function in 
\eqref{eq:pdf2pt} is a strictly positive definite function. As a consequence of Theorem \ref{thm:main}, we can state the 
following result:  

\begin{thm} \label{thm:2ptSPDF}
For $0 < t \le \pi$, the function $\t \mapsto (t-\t)_+^\d $ is strictly positive definite on $\PP^d(\RR)$ for $d =3,4,...$,
$\PP^d(\CC)$ for $d=4,6,...$, and $\PP^d (\HH)$ for $d =8,10, ...$ if $\d \ge \lceil \frac{d+1}{2} \rceil$ and on $\PP^{16}(\mathrm{Cay})$ if 
$\d \ge 9$. 
\end{thm}

\section{Positive integrals of the Jacobi polynomials}
\setcounter{equation}{0}

We start with several properties of the integrals $F_n^{(\a,\b),\d}$ defined in \eqref{eq:Fn^J}. 

\begin{lem} \label{lem:reduct1}
If $F_n^{(\a,\b), \d}(t) \ge  0$, then $F_n^{(\a,\b), \s}(t) \ge 0$ for $\s > \d \ge 0$. Moreover, this relation
holds if both $\ge$ are replaced by $>$. 
\end{lem}

\begin{proof}
For $\d  > 0$, define the fractional integral,  
$$
  \CL^\d f :=  \frac{1}{\Gamma(\d)} \int_0^t (t- \t)^{\d-1} f(\t) d\t, 
$$ 
which is called the Riemann-Liouville integral (see the wikipedia article under this name). It is well-known, and 
easy to verify, that $\CL^{\d+\mu} = \CL^\d \CL^\mu$, which implies immediately that 
$F_n^{(\a,\b),\s} = c \CL^{\s-\d} F_n^{(\a,\b),\d}$ for $\s > \d$, where $c = \Gamma(\s+1)/\Gamma(\d+1)$,
and proves the statement.   
\end{proof} 

\begin{lem} \label{lem:reduct2}
If $F_m^{(\a,\b), \d}(t) \ge  0$ for $m =n$ and $m=n+1$, then $F_n^{(\a,\b+1), \d}(t) \ge 0$ for $\b \ge 0$. 
Moreover, this relation holds if both $\ge$ are replaced by $>$. 
\end{lem}

\begin{proof}
We shall need the following identity \cite[(22.7.16)]{AS}, 
$$
 \left(n+\tfrac{\a+\b}{2}+1\right) (1+x) P_n^{(\a,\b+1)}(x) = (n+\b+1) P_n^{(\a,\b)}(x) + (n+1) P_{n+1}^{(\a,\b)}(x),
$$
which implies immediately that 
\begin{align} \label{eq:Jacobi1}
 & \left(\cos \tfrac{\t}{2}\right)^2 P_n^{(\a-\f12, \b+\f12)}(\cos \t) \\
  & \qquad\qquad  = A_n^{\a,\b} P_n^{(\a-\f12,\b-\f12)}(\cos \t)+B_n^{\a,\b}
   P_{n+1}^{(\a-\f12,\b-\f12)}(\cos \t), \notag
\end{align}
where 
$$
  A_n^{\a,\b} = \frac{n+\b+1/2}{2 n+ \a+\b + 1} \quad \hbox{and} \quad  B_n^{\a,\b} = \frac{n+1}{2 n+ \a+\b + 1}.
$$
Consequently, we obtain the relation
\begin{equation} \label{eq:Fn-reduc1}
  F_n^{(\a,\b+1),\d}(t) = A_n^{\a,\b}  F_n^{(\a,\b),\d}(t) + B_n^{\a,\b} F_{n+1}^{(\a,\b),\d}(t),  
\end{equation}
from which the statement of the lemma follows since $A_n^{\a,\b}$ and $B_n^{\a,\b}$ are positive numbers.
\end{proof}

\begin{lem} \label{lem:reduct3}
For $\a \ge 0$ and $t > 0$, 
\begin{equation} \label{eq:Fn-QT}
F_n^{(\a,0), \d}(t) = 2^{2\a+\d+1} a_n^\a F_{2n}^{(\a,\a), \d}(t/2),
 \qquad a_n^\a:= \frac{ (2n)!(\a+\f12)_n}{n! (\a+\f12)_{2n}}. 
\end{equation}
\end{lem}

\begin{proof}
Setting $x = \cos \t/2$ in the quadratic transformation \cite[(4.1.5)]{Sz} of the Jacobi polynomials, we obtain
\begin{equation} \label{eq:Jacobi2}
   P_n^{(\a-\f12,-\f12)}(\cos \t) = a_n^\a P_{2n}^{(\a-\f12, \a-\f12)}(\cos \tfrac{\t}{2}).
\end{equation}
Together with $\sin \t = 2 \sin \f{\t}{2} \cos \f{\t}2$, this identity implies immediately that 
\begin{align*}
   F_n^{(\a,0),\d}(t) & = a_n^\a 2^{2\a}  \int_0^t (t-\t)^\delta P_{2n}^{(\a-\f12, \a-\f12)}\left(\cos \tfrac{\t}{2}\right) (\sin \tfrac{\t}{2} \cos \tfrac{\t}2)^{2\a} d\t \\ 
 & = a_n^\a 2^{2\a+1}  \int_0^{t/2} (t - 2 \t)^\delta P_{2n}^{(\a-\f12, \a-\f12)}\left(\cos \t \right) (\sin \t)^{2\a}(\cos \t)^{2\a} d\t \notag \\ 
 & = a_n^\a 2^{2\a+\d+1}  F_{2n}^{(\a,\a),\d}(t/2), \notag
\end{align*}
which is \eqref{eq:Fn-QT}.
\end{proof}

We need a generalization of \eqref{eq:Fn-QT}. For $m \in \NN_0$ and $\a, \b  > -\f12$, define 
$$
  F_{n,m}^{(\a, \b),\d}(t) = \int_0^t (t-\t)^\delta P_{2^m n}^{(\a-\f12, \b-\f12)}\left(\cos \tfrac{\t}{2^m}\right)
      \left(\sin \tfrac{\t}{2^{m+1}}\right)^{2\a} \left(\cos \tfrac{\t}{2^{m+1}}\right)^{2\b} d\t.
$$
Evidently, $F_{n,0}^{(\a, \b),\d}(t) = F_{n}^{(\a,\b),\d}(t)$. 

\begin{lem}
For $n\in \NN_0$ and $m \in \NN$, 
\begin{align} \label{eq:Fnm-recur}
   F_{n,m}^{(\a,0),\d}(t) = a_{2^m n}^\a 2^{2\a+\d+1} F_{n,m+1}^{(\a,\a),\d}(t). 
\end{align}
\end{lem}
\begin{proof}
The change of variable $\t \mapsto 2^m \t$ in the integral shows that 
$$
  F_{n,m}^{(\a,\b),\d}(t) = 2^{m(\d+1)} F_{2^m n}^{(\a,\b),\d}(t/2^m). 
$$
By \eqref{eq:Fn-QT}, $F_n^{(\a,0),\d}(t) = a_n^\a 2^{2\a+\d+1} F_{n,1}^{(\a,\a),\d}(t)$. For $m \ge 1$, 
\eqref{eq:Fnm-recur} is established by the same deduction based on \eqref{eq:Jacobi2}.
\end{proof}

We will also need a result on the integrals of the Bessel functions. For $\a > 0$, the Bessel function of index $\a$ 
is defined by
$$
   J_\a(z) = \left(\frac{z}{2} \right)^\a  \sum_{k=0}^\infty \frac{(-1)^k }{k! \Gamma (k+\a+1)} \left(\frac{z}{2}\right)^{2k}.
$$
The following theorem is established in \cite[Theorem 6.1]{GG}. 

\begin{thm} \label{thm:IntBessel}
If $0\le \mu \le 1$ and $\a+ \mu \ge 1/2$, then the inequality 
\begin{equation*}
   \int_0^x (x-t)^{\a+2\mu-1/2} t^{\a+\mu} J_\a(t) dt \ge 0, \qquad x > 0, 
\end{equation*}
holds. Moreover, equality occurs only when $\mu =0$, $\a = -1/2$ or $\mu =1$, $\a = -1/2$. 
\end{thm}

The Bessel function is related to the Jacobi polynomials by \cite[Theorem 8.1]{Sz} 
\begin{equation}\label{eq:Jacobi-Bessel}
   \lim_{m \to \infty}m^{-\a} P_m^{(\a,\b)}\left(\cos \f{z}{m}\right) =  \left(\f {z}{2}\right)^{-\a} J_\a(z),
\end{equation}
and the limit holds uniformly in every bounded region of the complex $z$-plane. This relation will help us to deduce the
positivity of the integrals of the Jacobi polynomials from that of the integrals of Bessel functions. 

We are now ready to prove Theorem \ref{thm:main}, which we restate below. 

\begin{thm} \label{thm:main2}
If $\a, \b \in \NN_0$ are not both zero, then $F_n^{(\a,\b),\d}(t) > 0$ for $t > 0$ and $n \in \NN_0$ if $\d \ge \a+1$. Moreover,
$F_n^{(0,0),\d}(t) \ge 0$ for  $t > 0$ and $n \in \NN_0$ if $\d \ge 1$.
\end{thm}

\begin{proof}
By Lemma \ref{lem:reduct1}, we only need to establish the positivity of $F_n^{(\a,\b),\d}$ for $\d = \a +1$.
First we assume $\a > 0$. Since the $\delta$ index is independent of $\b$ and $\b$ is a nonnegative integer, 
it follows from Lemma \ref{lem:reduct2} that it suffices to show that $F_n^{(\a, 0),\a+1}(t) > 0$ for all $t > 0$ 
and $n \in \NN_0$, which is equivalent to, by \eqref{eq:Fnm-recur}, $F_{n,1}^{(\a,\a),\a+1}(t) > 0$. Since $\a \in \NN$, 
we can apply Lemma \ref{lem:reduct2} again to deduce the problem to showing that $F_{n,1}^{(\a,0), \a+1}(t) > 0$, which 
is equivalent to, by \eqref{eq:Fnm-recur}, $F_{n,2}^{(\a,\a),\d}(t) > 0$. Evidently we can repeat this process to conclude 
that it is sufficient to prove $F_{n,m}^{(\a,0), \a+1} (t) > 0$ for all $t > 0$ for a sufficiently large integer $m$. Indeed, if 
$F_{n,m}^{(\a,0),\a+1}(t) >0$ for all $n$, the above argument shows that 
$F_{n}^{(\a,\b),\a+1}(t) \ge a_n F_{n,m}^{(\a,0),\a+1}(t) >0$, where $a_n$ is a positive constant. 

Taking the limit $m \to \infty$, it follows from \eqref{eq:Jacobi-Bessel} and $\lim_{x \to 0} \sin x/x =0$ that 
\begin{align*}
  \lim_{m\to \infty} 2^{m(\a+\f12)} P_{2^m n}^{(\a-\f12,-\f12)}\left(\cos\tfrac{\t}{2^m} \right) \left(\sin \tfrac{\t}{2^m}\right)^{2\a} 
  &  =  n^{\a-\f12} \t^{2\a}  \lim_{m\to \infty} \frac{P_{2^m n}^{(\a-\f12,-\f12)}\left(\cos\tfrac{\t}{2^m} \right)}{(2^m n)^{\a-\f12}} \\
  &  = (2 n )^{\a-\f12} \t^{\a+\f12} J_{\a-\f12}(n \t) 
\end{align*}
and the limit holds uniformly on $[0, M]$ for every fixed $M > 0$. Consequently, fix $n \in \NN$ and choose $M=  n \pi$, it 
follows that
\begin{align*}
 \lim_{m\to\infty} 2^{m(\a + \f12)} F_{n,m}^{(\a,0),\d}(t) & = (2n)^{\a-\f12} \int_0^t (t-\t)^\d   \t^{\a +\f12} J_{\a -\f12}(n \t) d\t \\
    & = \frac{2^{\a-\f12}}{n^{\d+2}}  \int_0^{n t} (n t-\t)^\d   \t^{\a +\f12} J_{\a -\f12}(\t) d\t.
\end{align*}
By Theorem \ref{thm:IntBessel} with $\mu = 1$, the last integral is strictly positive for all $t$ if $\d = \a+1$ and $\a > 0$. 
This completes the proof for the case $\a > 0$. 

Next we consider the case $\a =0$ and $\b \in \NN$. Since $\b$ is a positive integer, applying \eqref{eq:Jacobi1} 
repeatedly shows that
$$
   F_n^{(0,\b),\d}(t) = a_n  F_n^{(0,0),\d}(t) + \ldots  + a_{n+\b} F_{n+\b}^{(0,0),\d}(t) 
$$
for some positive numbers $a_n, \ldots, a_{n+\b}$. We need to show that this is strictly positive when $\d =1$. However,
it is known \cite[(4.1.7)]{Sz} that $P_n^{(-\f12,-\f12)}(\cos \t)= (\f12)_n /n! \cos n \t$, so that 
$$
F_n^{(0,0),1}(t) = \f{(\f12)_n}{n!} \int_0^t (t-\t) \cos n \t d\t = 2 \f{\sin^2 (n t)}{n^2}. 
$$
It is evident that $F_n^{(0,0),1}(t)$ and $F_{n+1}^{(0,0),1}(t)$ cannot be both zero for the same $t$ whenever $n \in \NN_0$.
Consequently, since $\b \ge 1$, we see that $F_n^{(0,\b),1} (t) > 0$ for all $t > 0$. 

The last case that $F_n^{(0,0),\d} (t) \ge 0$ follows immediately from the above identity and Lemma \ref{lem:reduct1}. 
The proof is completed. 
\end{proof}

To show that the function $\t \mapsto (t-\t)_+^\d$ is a positive definite function, we need to verify that its Fourier-Jacobi
expansion converges at $\t =0$. By the orthogonality of the Jacobi polynomials, the expansion is given by
\begin{equation}\label{eq:FourierJacobi}
   (t-\t)_+^\d = \sum_{n=0}^\infty \frac{F_n^{(\a-\f12,\b-\f12),\d}(t) }{h_n^{\a-\f12,\b-\f12}}P_n^{(\a-\f12,\b-\f12)}(\cos \t), 
\end{equation}
where the convergence holds in weighted $L^2$ space and 
$$
   h_n^{(\a-\f12,\b-\f12)} = 2^{-\a-\b+1} \int_{-1}^1 |P_n^{(\a-\f12,\b-\f12)}|^2 (1-x)^{\a-\f12} (1+x)^{\b-\f12} dx.
$$

\begin{prop}
For $\a,\b \in \NN_0$ with $\max \{\a,\b\} >0$, the Fourier-Jacobi expansion \eqref{eq:FourierJacobi} of the function 
$\t \mapsto (t-\t)_+^\d$ converges at $\t =0$ when $\d \ge \a+1$. 
\end{prop}

\begin{proof}
By \cite[(4.1.4) and (4.3.3)]{Sz}, it is easy to see that $P_n^{(\a-\f12,\b-\f12)}(1) /h_n^{(\a-\f12,\b-\f12)} \sim n^{\a+\f12}$. 
Thus, it is sufficient to show that 
\begin{equation} \label{eq:Fn-norm}
\|F_n^{(\a,\b),\d}\|_\infty \le c n^{-\a-5/2},  \qquad \delta \ge \a+1.
\end{equation}
For $\d > \a+1$, it follows readily from Lemma \ref{lem:reduct1} that $\|F_n^{(\a,\b),\d}\|_\infty \le \|F_n^{(\a,\b),\a+1}\|_\infty$,
so that we only need to consider $\d = \a+1$. Since $\b \in \NN_0$, using \eqref{eq:Fn-reduc1} repeatedly, we see that
$\|F_n^{(\a,\b),\a+1}\|_\infty \le  2^\b \max_{0 \le j\le \beta} \|F_{n+j}^{(\a,0),\a+1}\|_\infty$ for $n \ge 1$. Moreover, by
\eqref{eq:Fn-QT}, $\|F_{n}^{(\a,0),\a+1}\|_\infty \le c_{\a} \|F_{n}^{(\a,\a),\a+1}\|_\infty$, where $c_{\a}$ is 
a constant independent of $n$. Furthermore, by \eqref{eq:Cn=Pn2}, $\|F_n^{(\a,\a),\a+1}\|_\infty = O(1) n^{-\a+1/2} 
\|F_{n}^{\a,\a+1}\|_\infty$, where $F_{n}^{\a,\a+1}$ is defined in \eqref{eq:FnG}.  Finally,  by \cite[Lemma 3.6]{BCX}, 
$\|F_{n}^{\a,\a+1}\| = O(1) n^{-3}$, which completes the proof of \eqref{eq:Fn-norm}. 
\end{proof}

We now state further properties of $F_n^{(\a,\b),\d}$ and use them to prove Theorem \ref{thm:main} when 
$\a$ and $\b$ are non-integers. 

\begin{lem} \label{lem:translat}
Let $\a, \b, \g > -1$. Assume $\frac{\g-1}{2} < \a < \g$. Then
\begin{enumerate}[ \quad  1. ]
\item If $F_n^{(\g,\b), \d} (t) >0$ for all $n \in \NN_0$, then $F_n^{(\a,\b),\d}(t) >0$ for all $n \in \NN_0$.
\item If $F_n^{(\g,\g),\d}(t) > 0$ for all $n \in \NN_0$, then $F_n^{(\a,\a),\d}(t) > 0$ for all $n \in \NN_0$.
\end{enumerate}
\end{lem}

\begin{proof}
The first item follows from the identity \cite[(3.41)]{AF}
$$
   (1-x)^{\a} P_n^{(\a-\f12,\b-\f12)}(x) = \sum_{k=n}^\infty a_{k,n}^{\a,\b,\g} (1-x)^{\g} P_k^{(\g-\f12,\b-\f12)}(x)
$$
where $\g > \a$ and 
\begin{align*}
a_{k,n}^{\a,\b,\g} = \frac{\Gamma(n+\a+\f12)\Gamma(n+k+\b+\g)\Gamma(k-n+\g-\a) k! (2k+\g+\b)}
     {n! (k-n)! \Gamma(n+k+\a+\b+1) \Gamma(k+\g+\f12) \Gamma(\g-\a) 2^{\g-\a}},
\end{align*}
and the series converges when $\a > \f{\g-1}{2}$. Indeed, setting $x = \cos \t$, so that $1-x = 2 \sin^2 \f{\t}{2}$, we obtain
$$
  F_n^{(\a,\b),\d}(t) = \sum_{k=n}^\infty  a_{k,n}^{\a,\b,\g} 2^{\g-\a} F_n^{(\g,\b),\d}(t), 
$$
which yields the first item since $a_{k,n}^{\a,\b,\g} > 0$ for $\frac{\g-1}{2} < \a < \g$. Similarly, the second item follows from
the identity \cite[(3.43)]{Sz}, 
$$
   (1-x^2)^{\a} P_n^{(\a-\f12,\a-\f12)}(x) = \sum_{k=n}^\infty b_{k,n}^{\a,\g} (1-x^2)^{\g} P_n^{(\g-\f12,\g-\f12)}(x),
$$
stated in terms of the Gegenbauer polynomials and rewritten using \eqref{eq:Cn=Pn}, where 
$$
  b_{k,n}^{\a,\g} =  \frac{(n+2k)!\Gamma(n+k+\g)\Gamma(k+\g-\a)\Gamma(n+\a+\f12) (n+2k+\g)}
     {n! k! \Gamma(\g-\a) \Gamma(n+k+\a+1) \Gamma(n+2k+\g+\f12)} 
$$
for $\a > \f{\g-1}{2}$. Writing $\sin \t = 2 \sin \f{\t}{2} \cos \f{\t}2$, we then obtain 
$$
  F_n^{(\a,\a),\d}(t) = \sum_{k=n}^\infty  b_{k,n}^{\a,\g} 2^{2 \g- 2 \a} F_n^{(\g,\g),\d}(t), 
$$
which yields the second item since $b_{k,n}^{\a,\g} > 0$ for $\a < \g$. 
\end{proof}

\begin{thm}\label{thm:main3}
Let $\a \in \RR \setminus \NN_0$, $\b \in \RR$ and $\b \ge 0$. If $\d \ge  \lceil \a \rceil  +1$, then $F_n^{(\a,\b),\d}(t) > 0$ 
for all $t > 0$ and $n \in \NN_0$   
when 
\begin{enumerate}[ \quad  1.]
\item $\a > 0$ and $\b \in \NN_0$;
\item $\a = \b > 0$;  
\item $0\le  \lfloor \b \rfloor < \a \le  \b$.  
\end{enumerate} 
\end{thm}

\begin{proof}
Let $\b \in \NN_0$ and $\a  >0$. By Theorem \ref{thm:main}, $F_n^{(\lceil \a \rceil,\b),\a+1} (t) > 0$ for $t > 0$ and $n\in \NN_0$,
from which the first item follows when we apply the first item in Lemma \ref{lem:translat}, which holds since  
$\f{\lfloor \a \rfloor}2= \f{\lceil \a \rceil -1}{2} < \a < \lceil \a \rceil$. Starting from the case $\a = \b \in \NN$, the second item is a 
direct consequence of the second item in Lemma \ref{lem:translat}. For the third item, we only need to consider the case when 
$\b$ is not a positive integer. Starting from $F_n^{(\b,\b),\d+1}(t) > 0$ for $\d = \lceil \b \rceil+1$ and applying the first item in 
Lemma \ref{lem:translat}, we see that $F_n^{(\a,\b),\d}(t) > 0$ for $\lfloor \b \rfloor < \a < \b$. Since $\b$ is not an integer, 
we have  $\lceil \b \rceil = \lceil \a \rceil$, so that $\d \ge \lceil \a \rceil+1$. This completes the proof. 
\end{proof}

As shown in \cite{BCX}, Theorem \ref{thm:main} immediately implies Corollary \ref{cor:SPDF} when the dimension $d$ is
an even integer. For $d$ is an odd integer, the proof of \cite{BCX} uses the fact that a strictly positive definite function
on $\sph$ is necessarily a strictly positive definite function on $\SS^{d-2}$. A direct proof for odd $d$ follows from 
Theorem \ref{thm:main3}, which also gives the proof of Theorem \ref{thm:2ptSPDF}. 

We state the case of $\a = \b$ in Theorem \ref{thm:main3} in terms of $F_n^{\l,\d}$, defined in \eqref{eq:Fn^J}, as a corollary.  

\begin{cor}
For $\l \in \RR$ and $\l > 0$, $F_n^{\l,\d}(t) > 0$ for $0 < t \le \pi$ and $n \in \NN_0$ if $\d \ge \lceil \l \rceil + 1$. 
\end{cor}

Comparing with Conjecture \ref{conjecture}, one may ask if the condition $\d \ge \lceil \l \rceil + 1$ can be improved to
$\d \ge \l+1$; that is, if the conjecture holds positively for all $\l > 0$. If it held, say for $\l$ being a half integer, then
the function $\t \mapsto (t-\t)_+^\d$ would be strictly positive on $\sph$ if $\d > \f{d}{2}$ for $d$ being odd integers as well. 
A more general conjecture is the following: 

\begin{conj}
For $\a,\b \in (0, \infty) \setminus \NN$, $F_n^{(\a,\b),\d}(t) > 0 $ for all $t\in (0,\pi]$ and $n \in \NN_0$ 
if $\d \ge \a + 1$. 
\end{conj} 

We end with two more remarks on the integrals of the Jacobi polynomials. 

Using the symmetry of the Jacobi polynomials,
$P_n^{(\a,\b)}(-x) = (-1)^n P_n^{(\b,\a)}(x)$, we can deduce a companion of \eqref{eq:Fn-reduc1} that states
\begin{equation} \label{eq:Fn-reduc2}
  F_n^{(\a+1,\b),\d}(t) = A_n^{\b,\a}  F_n^{(\a,\b),\d}(t) - B_n^{\a,\b} F_{n+1}^{(\a,\b),\d}(t). 
\end{equation}
Together with \eqref{eq:Fn-reduc1}, this also implies the following identity: 
\begin{equation} \label{eq:Fn-reduc3}
  F_n^{(\a+1,\b),\d}(t)+ F_n^{(\a,\b+1),\d}(t) = F_{n}^{(\a,\b),\d}(t). 
\end{equation}
Although these two identities are not needed in our proof, they have interesting implications. We give one
example. 

\begin{prop}
Let $\a > 0$ and $\b \in \NN_0$. If $\d \ge \lceil \a \rceil +2$ and $\a > 0$, then $F_n^{(\a,\b),\d}(t) > F_n^{(\a,\b+1),\d}(t)$ 
for $0< t\le \pi$ and $n \in \NN_0$.
\end{prop}

\begin{proof}
The condition implies $F_n^{(\a+1,\b),\d}(t)> 0$ by Theorem \ref{thm:main3}, so that the stated result follows 
from \eqref{eq:Fn-reduc3}. 
\end{proof}

When $\d$ is an integer and $\l$ is an integer, the integral $F_n^{\l,\d}$ defined in \eqref{eq:FnG} can 
be written as a finite sum, as shown in the following lemma \cite[Lemma 3.2]{BCX}. 

\begin{lem}\label{lem:sin-cos}
For $\mu =1,2,3,...$, 
\begin{align*}
 & F_n^{2\mu-1,2\mu}(t) =\sum_{k=0}^{2 \mu-1} b_{k,n}^{2\mu-1} 
   \frac{(-1)^{\mu} (2\mu)!}{(n+2k)^{2\mu}} \left[ \sin (n+2k) t - \sum_{j=0}^{\mu-1} (-1)^j 
       \frac{(n+2k)^{2j+1}}{(2j+1)!}  t^{2j+1} \right], \\
&  F_n^{2\mu, 2\mu+1}(t) = \sum_{k=0}^{2 \mu} b_{k,n}^{2\mu} 
   \frac{(-1)^{\mu+1} (2\mu+1)!}{(n+2k)^{2\mu+1}} \left[ \cos (n+2 k) t - \sum_{j=0}^\mu (-1)^j 
  \frac{(n+2k)^{2j}}{(2j)!} t^{2j} \right],
\end{align*}
where 
\begin{equation*}
b^\mu_{k,n} := \frac{ 2^{1-2\mu}}{\Gamma(\mu)}  (- 1)^k \binom{\mu}{k}  \frac{ (n+1)_{2\mu-1}  }{(n+k)_{\mu+1}}.
\end{equation*}
\end{lem}

The square brackets in the above formulas are the difference of $\sin(n+2k) t$ and its Taylor polynomial and
$\cos (n+2k) t$ and its Taylor polynomial, respectively. Both these sums are positive for all $t >0$ by
Theorem \ref{thm:main}.

\medskip\noindent
{\it Acknowledgement:} The author thanks two anonymous referees for their careful readings and corrections, and for
their helpful suggestions.

\end{document}